\documentclass[11pt,a4paper]{amsart}
\usepackage{amssymb,color}
\usepackage[english]{babel}          
\usepackage{verbatim}
\usepackage[T1]{fontenc}             

\usepackage{floatflt,graphicx}
\usepackage{color,xcolor}
\usepackage{subfig}
\usepackage{enumerate}
\usepackage{tikz}
\usepackage{animate}
\usepackage{amsmath}
\usepackage{amsthm}
\usepackage{epstopdf}
\usepackage[utf8]{inputenc} 

\usepackage{geometry}
\newcounter{minutes}
\divide\time by 60
\newcounter{hours}
\multiply\time by 60 \addtocounter{minutes}{-\time}
\title
[Triangular Ratio Metric Under Quasiconformal Mappings In Sector Domains]
{Triangular Ratio Metric Under Quasiconformal Mappings In Sector Domains}

\author{Oona Rainio}
\address{Department of Mathematics and Statistics\\
University of Turku\\ Turku, Finland}
\email{ormrai@utu.fi}

\author{Matti Vuorinen}
\address{Department of Mathematics and Statistics\\
University of Turku\\ Turku, Finland}
\email{vuorinen@utu.fi}

\keywords{Hyperbolic metric, hyperbolic type metrics, triangular ratio metric, quasiconformal maps, sector domain}\subjclass[2010]{51M10, 30C65}

\dedicatory{}

\theoremstyle{plain}
\newtheorem{theorem}[equation]{Theorem}
\newtheorem{corollary}[equation]{Corollary}
\newtheorem{lemma}[equation]{Lemma}

\newtheorem{proposition}[equation]{Proposition}

\theoremstyle{definition}

\theoremstyle{remark}
\newtheorem{remark}[equation]{Remark}

\numberwithin{equation}{section}

\newcommand{\R}{\mathbb{R}}

\newcommand{\C}{\mathbb{C}}

\newcommand{\B}{\mathbb{B}}

\newcommand{\uhp}{\mathbb{H}}

\begin{document}

\def\thefootnote{}
\footnotetext{ \texttt{File:~sqm20200525.tex,
           printed: \number\year-\number\month-\number\day,
           \thehours.\ifnum\theminutes<10{0}\fi\theminutes}
} \makeatletter\def\thefootnote{\@arabic\c@footnote}\makeatother

\begin{abstract}
Hyperbolic metric and different hyperbolic type metrics are studied in open sector domains of the complex plane. Several sharp inequalities are proven for them. Our main result describes the behavior of the triangular ratio metric under quasiconformal maps from one sector onto another one.
\end{abstract}

\maketitle
\section{Introduction}
\setcounter{equation}{0}

Geometric function theory studies families of functions such as conformal maps, analytic functions as well as quasiconformal and
quasiregular mappings defined in subdomains $G$ of $\R^n, n\ge 2\,.$ In this research, a key notion is an {\it intrinsic distance}, which is a distance between two points in the domain, specific to the domain itself and, in particular, its boundary \cite{gh, hkvbook}. In the planar case $n=2$, such a distance is the hyperbolic distance that can be readily defined by use of a conformal mapping for a simply connected domain, but this does not generalize to higher dimensions. It is natural therefore to look for various extensions and generalizations of hyperbolic metrics. Twelve metrics recurrent in geometric function theory are listed by A. Papadopoulos in \cite[pp. 42-48]{p14}.

Many people have studied generalizations of hyperbolic metrics to subdomains of $\R^n$, $n\geq3$, and found hyperbolic type metrics, which share some but not all properties of the hyperbolic metric. 
In their study of quasidisks, F.W. Gehring and K. Hag  \cite{gh}  apply 
 the hyperbolic, quasihyperbolic, distance ratio, and Apollonian metrics. Very recently, the geometry
of the quasihyperbolic metric has been studied by D. Herron and P. Julian \cite{hj}, A. Rasila et al. \cite{rtz}, S. Buckley and D. Herron \cite{bh}.
Another hyperbolic type metric is  the triangular ratio metric introduced by P. H\"ast\"o \cite{h} and most recently studied by M. Fujimura et al. \cite{fmv}. The interrelations between these metrics have been investigated by P. H\"ast\"o, Z. Ibragimov, D. Minda, S. Ponnusamy and S. Sahoo \cite{himps}. See also D. Herron et al. \cite{him}, and A. Aksoy et al. \cite{aiw}.

Our work is motivated by the recent progress of the study of intrinsic geometry of domains, of which
the above papers and the monographs \cite{gh, hkvbook, p14} are examples. First, in Section \ref{sector_section}, we find new inequalities between three different hyperbolic type metrics in sector domains of the complex plane and establish sharp forms of some earlier results in \cite{chkv,hvz}. In Section \ref{rhos_section}, we apply a rotation method involving M\"obius transformations to obtain a sharp inequality between the triangular ratio metric and the hyperbolic metric in a sector with a fixed angle. Finally, in Section \ref{s_in_qms_section}, we present our main result that provides a sharp distortion theorem for the triangular ratio metric under quasiconformal maps between two planar sector domains.
\section{Preliminary facts}

Define the following hyperbolic type metrics:
The triangular ratio metric $s_G:G\times G\to\R$,
\begin{align*}
s_G(x,y)=\frac{|x-y|}{\inf_{z\in\partial G}(|x-z|+|z-y|)}, 
\end{align*}
the $j^*_G$-metric $j^*_G:G\times G\to\R,$
\begin{align*}
j^*_G(x,y)=\frac{|x-y|}{|x-y|+2\min\{d_G(x),d_G(y)\}},    
\end{align*}
and the point pair function $p_G:G\times G\to\R,$
\begin{align*}
p_G(x,y)=\frac{|x-y|}{\sqrt{|x-y|^2+4d_G(x)d_G(y)}},    
\end{align*}
These functions were studied in \cite{chkv, hvz}. Here, the domain $G$ is a non-empty, open, proper and connected subset of $\R^n$ and the notation $d_G(x)$ means the Euclidean distance dist$(x,\partial G)=\inf\{|x-z|\text{ }|\text{ }z\in\partial G\}$ between the point $x$ and the boundary of $G$. Note that the point pair function is not always a metric \cite[Rmk 3.1 p. 689]{chkv}.

In this paper, we especially focus on the case where the domain $G$ is an open sector $S_\theta=\{x\in\C\text{ }|\text{ }0<\arg(x)<\theta\}$ with an angle $\theta\in(0,2\pi)$. In the limiting case $\theta=0$, we consider the strip domain $S_0=\{x\in\C\text{ }|\text{ }0<\text{Im}(x)<\pi\}$. Other common domains are the upper half-plane $\uhp^n=\{(x_1,...,x_n)\in\R^n\text{ }|\text{ }x_n>0\}$ and the unit ball $\B^n=B^n(0,1)$. Here, $B^n(x,r)$ is the open ball with the Euclidean metric, $\overline{B}^n(x,r)$ is the corresponding closed ball and $S^{n-1}(x,r)$ is the boundary $\partial B^n(x,r)$. For two distinct points $x,y$, $L(x,y)$ is the Euclidean line passing through them.

With the notations presented above, we can also write the formulas for the hyperbolic metric
\begin{align*}
\text{ch}\rho_{\uhp^n}(x,y)&=1+\frac{|x-y|^2}{2d_{\uhp^n}(x)d_{\uhp^n}(y)},\quad x,y\in\uhp^n,\\
\text{sh}^2\frac{\rho_{\B^n}(x,y)}{2}&=\frac{|x-y|^2}{(1-|x|^2)(1-|y|^2)},\quad x,y\in\B^n
\end{align*}
in the upper half-plane and in the Poincaré unit disk, respectively \cite[(2.8) p. 15]{bm}. Both of these metrics are invariant under a Möbius transformation $h: G\to D$ where $G,D\in\{\B^n,\uhp^n\}\,,$ in other words, $\rho_D(h(x),h(y))= \rho_G(x,y)$ for all $x,y\in G$. In the two-dimensional plane, the definitions of hyperbolic metric can be simplified to
\begin{align*}
\text{th}\frac{\rho_{\uhp^2}(x,y)}{2}&=\text{th}\left(\frac{1}{2}\log\left(\frac{|x-\overline{y}|+|x-y|}{|x-\overline{y}|-|x-y|}\right)\right)=\left|\frac{x-y}{x-\overline{y}}\right|,\\
\text{th}\frac{\rho_{\B^2}(x,y)}{2}&=\text{th}\left(\frac{1}{2}\log\left(\frac{|1-x\overline{y}|+|x-y|}{|1-x\overline{y}|-|x-y|}\right)\right)=\left|\frac{x-y}{1-x\overline{y}}\right|,
\end{align*}
where $\overline{y}$ is the complex conjugate of $y$. Moreover, in a planar simply-connected
domain, the hyperbolic metric can be defined in terms of a conformal mapping of the domain onto
the unit disk because the hyperbolic metric is invariant under conformal mappings \cite[Thm 6.3 p. 26]{bm}.

The following inequalities between hyperbolic type metrics are already known:

\begin{theorem}\label{jp_bounds}
\emph{\cite[Lemma 2.3, p. 1125]{hvz}} For a proper subdomain $G$ of $\R^n$, the inequality $j^*_G(x,y)\leq p_G(x,y)\leq\sqrt{2}j^*_G(x,y)$ holds for all $x,y\in G$.
\end{theorem}

\begin{theorem}\label{sj_properbounds}
\emph{\cite[Lemma 2.1, p. 1124 \& Lemma 2.2, p. 1125]{hvz}} For a proper subdomain $G$ of $\R^n$, the inequality $j^*_G(x,y)\leq s_G(x,y)\leq2j^*_G(x,y)$ holds for all $x,y\in G$.
\end{theorem}

\begin{theorem}\label{sj_convexbounds}
\emph{\cite[Lemma 2.8 \& Thm 2.9(1), p. 1129]{hvz}} If $G\subsetneq\R^n$ is convex, $j^*_G(x,y)\leq s_G(x,y)\leq\sqrt{2}j^*_G(x,y)$ holds for all $x,y\in G$.
\end{theorem}

The following results are useful when calculating the value of $s_G(x,y)$:

\begin{lemma}\label{heron_lemma}
(Heron's shortest distance problem) Given $x,y\in\uhp^2$, the Heron point $w=L(\overline{x},y)\cap\R$
minimizes the sum $|x-z|+|z-y|$ where $z\in\R$, and therefore $\inf_{z\in\R}(|x-z|+|z-y|)=|\overline{x}-y|$.
\end{lemma}

It follows from Lemma \ref{heron_lemma} that for all $x,y\in\uhp^2$ 
\begin{align*}
s_{\uhp^2}(x,y)=\left|\frac{x-y}{\overline{x}-y}\right|.    
\end{align*}

\begin{theorem}
For all $\theta\in(0,2\pi)$ and $x,y\in S_\theta$, there is an analytical solution to the value of $s_{S_\theta}(x,y)$.
\end{theorem}
\begin{proof}
Consider a line $l\subset\R^n$, a half-line $l_0\subset l$ and two points $x,y\in\R^n$. Let $x'$ be the point $x$ reflected over the line $l$. If $[x,y]\cap l_0\neq\varnothing$, $\inf_{z\in l_0}(|x-z|+|z-y|)=|x-y|$. If $[x,y]\cap l=\varnothing$ and $[x,x']\cap l_0\neq\varnothing$, then, by Lemma \ref{heron_lemma}, $\inf_{z\in l_0}(|x-z|+|z-y|)=|x'-y|$. Otherwise, $\inf_{z\in l_0}(|x-z|+|z-y|)=|x-z_0|+|z_0-y|$ where $z_0$ is the endpoint of $l_0$. Clearly, 
\begin{align*}
\inf_{z\in\partial S_\theta}(|x-z|+|z-y|)=\min\{\inf_{z\in l_i}(|x-z|+|z-y|)\text{ }|\text{ }i=1,2\},    
\end{align*}
where $l_1,l_2$ are the half-lines forming the sector $S_\theta$.
\end{proof}

\section{Hyperbolic Type Metrics}\label{sector_section}

In this section, our main result is Theorem \ref{fin_thm_sjp_sector}. This theorem provides sharp inequalities between the hyperbolic type metrics in a sector domain. First, we will show that certain equalities are possible.

\begin{lemma}\label{sjp_equ_possible}
For any fixed domain $G\subsetneq\R^n$, there are distinct points $x,y\in G$ such that $s_G(x,y)=j^*_G(x,y)=p_G(x,y)$.
\end{lemma}
\begin{proof}
Fix $x\in G$ and choose a ball $B^n(x,r)\subset G$ with $S^{n-1}(x,r)\cap\partial G\neq\varnothing$ where $r>0$. Fix $z\in S^{n-1}(x,r)\cap\partial G$ and $y\in[x,z]$ so that $|x-y|=kr$ with $k\in(0,1)$. From $d_G(x)=r$ and $d_G(y)=(1-k)r$, it follows that $j^*_G(x,y)=p_G(x,y)=k\slash(2-k)$.
By \cite[(1), p. 205]{hkvbook},
\begin{align*}
s_G(x,y)\leq s_{B^n(x,r)}(x,y)=s_{\B^n}(0,k)=\frac{k}{2-k}=\frac{|x-y|}{|x-z|+|z-y|}\leq s_G(x,y)
\end{align*}
so that $s_G(x,y)=k\slash(2-k)$, too.
\end{proof}

\begin{lemma}\label{jp_sharproot}
For any fixed domain $G\subsetneq\R^2$, there are distinct points $x,y\in G$ such that $p_G(x,y)=\sqrt{2}j^*_G(x,y)$, if there exists a line segment $[u,v]\subset\partial G$.
\end{lemma}
\begin{proof}
First, note that the equality $p_G(x,y)=\sqrt{2}j^*_G(x,y)$ holds non-trivially when $x\neq y$ and $|x-y|\slash2=d_G(x)=d_G(y)$. If $[u,v]\subset\partial G$, we can fix a segment $[x,y]\subset G$ so that it is parallel to the segment $[u,v]$, has the same perpendicular bisector as $[u,v]$ and fulfills $d_G(x)=d_G(y)=d([x,y],[u,v])=|x-y|\slash2$. The equality $p_G(x,y)=\sqrt{2}j^*_G(x,y)$ follows.
\end{proof}

\begin{remark}\label{invariant_metrics}
The metrics $s_G$ and $j^*_G$, and the point pair function $p_G$ are all invariant under the stretching $z\mapsto rz$ by a factor $r>0$, if the domain $G$ is, for instance, $\uhp^2$, $S_\theta$ or $\R^n\backslash\{0\}$. 
\end{remark}

\begin{theorem}\label{thm_sj3}
For a fixed angle $\theta\in(\pi,2\pi)$ and for all $x,y\in S_\theta$, the sharp inequality $s_{S_\theta}(x,y)\leq2\sin(\theta\slash4)j^*_{S_\theta}(x,y)$ holds.
\end{theorem}
\begin{proof}
Consider the quotient
\begin{align}\label{sj_quo}
\frac{s_{S_\theta}(x,y)}{j^*_{S_\theta}(x,y)}=\frac{|x-y|+2\min\{d_{S_\theta}(x),d_{S_\theta}(y)\}}{\inf_{z\in\partial S_\theta}(|x-z|+|z-y|)}.   
\end{align}
We clearly need to find the maximal value of this quotient for each $\theta\in(\pi,2\pi)$. First, fix the point $z$ so that it gives the infimum $\inf_{z\in\partial S_\theta}(|x-z|+|z-y|)$. Since clearly $d_{S_\theta}(x)\leq|x-z|$ and $d_{S_\theta}(y)\leq|z-y|$, the sum $d_{S_\theta}(x)+d_{S_\theta}(y)$ has an upper limit $|x-z|+|z-y|$. If we set $r=d_{S_\theta}(x)+d_{S_\theta}(y)$, it follows that $\min\{d_{S_\theta}(x),d_{S_\theta}(y)\}\leq\min\{d_{S_\theta}(x),r-d_{S_\theta}(x)\}$. To find the maximum of $\min\{d_{S_\theta}(x),d_{S_\theta}(y)\}$, we need to choose $d_{S_\theta}(x)=r\slash2=d_{S_\theta}(y)$. 

Now, we need to minimize the value of $\inf_{z\in\partial S_\theta}(|x-z|+|z-y|)$ compared to the numerator $|x-y|+r$ of the quotient \eqref{sj_quo}. This happens when the point $z$ in the infimum is fixed to the origin, and $x$ and $y$ are on the different sides of the bisector of the sector $S_\theta$. It follows that $\inf_{z\in\partial S_\theta}(|x-z|+|z-y|)=|x|+|y|$. 

Let us yet consider the distance $|x-y|$ compared to the infimum in the quotient \eqref{sj_quo}. Suppose $\inf_{z\in\partial S_\theta}(|x-z|+|z-y|)=|x|+|y|$ is fixed. By Pythagoras' theorem, $|x-y|=\sqrt{|x|^2-l^2}+\sqrt{|y|^2-l^2}$, where $l=d([x,y],\{0\})$ so we see that, for a fixed sum $|x|+|y|$, the value of $|x-y|$ is at maximum when $|x|=|y|$.

Because of these observations and by Remark \ref{invariant_metrics}, we can fix $x=e^{hi}$ and $y=e^{(\theta-h)i}$ with some $(\theta-\pi)\slash2<h<\theta\slash2$. It follows that
\begin{align*}
&|x-y|=2\sin(\theta\slash2-h),\\
&d_{S_\theta}(x)=d_{S_\theta}(y)=
\begin{cases}
			\sin(h) & \text{ if }h<\pi\slash2,\\
            1 & \text{otherwise,}
\end{cases}\quad\text{and}\\
&\inf_{z\in\partial S_\theta}(|x-z|+|z-y|)=2.
\end{align*}
If $h<\pi\slash2$, the quotient \eqref{sj_quo} is $\sin(\theta\slash2-h)+\sin(h)$, which attains a maximum when $h=\theta\slash4$. Clearly, $(\theta-\pi)\slash2<\theta\slash4<\pi\slash2$ for all $\theta\in(\pi,2\pi)$. Thus, the maximum value of the quotient \eqref{sj_quo} with limitation $h<\pi\slash2$ is $2\sin(\theta\slash4)$.

If $h\geq\pi\slash2$ instead, the quotient \eqref{sj_quo} is $\sin(\theta\slash2-h)+1$. Within the limitation $h\geq\pi\slash2$, this quotient attains its maximum $1-\cos(\theta\slash2)$ when $h=\pi\slash2$. Since, for all $\theta\in(\pi,2\pi)$, $1-\cos(\theta\slash2)<2\sin(\theta\slash4)$, the true maximum value of the quotient is $2\sin(\theta\slash4)$ and the theorem follows.
\end{proof}

\begin{theorem}\label{sp_commonbounds}
For a domain $G\subsetneq\R^n$, the sharp inequality
\begin{align*}
\frac{1}{\sqrt{2}}p_G(x,y)\leq s_G(x,y)\leq
\sqrt{2}p_G(x,y)  
\end{align*}
holds for all $x,y\in G$.
\end{theorem}
\begin{proof}
According to \cite[11.16(1), p. 203]{hkvbook}, the inequality $p_G(x,y)\slash\sqrt{2}\leq s_G(x,y)$ holds for all $x,y\in G$ and this is sharp for $p_G(x,y)\slash\sqrt{2}=s_G(x,y)$ if $G=\{z\in\C\text{ }|\text{ }0<\text{Im}(z)<1\}$, $x=i\slash4$ and $y=3i\slash4$.

If we fix $z\in\partial G$ so that it gives the infimum $\inf_{z\in\partial G}(|x-z|+|z-y|)$,
\begin{align*}
\frac{s_G(x,y)}{p_G(x,y)}&=\frac{\sqrt{|x-y|^2+4d_G(x)d_G(y)}}{|x-z|+|z-y|}
\leq\frac{\sqrt{|x-y|^2+4|x-z||z-y|}}{|x-z|+|z-y|}\\
&\leq\frac{\sqrt{(|x-z|+|z-y|)^2+4|x-z||z-y|}}{|x-z|+|z-y|}
=\sqrt{1+\frac{4|x-z||z-y|}{(|x-z|+|z-y|)^2}}.
\end{align*}
It can be shown by differentiation that the quotient above attains its maximum value $\sqrt{2}$, when $|x-z|=|z-y|$. This proves the inequality $s_G(x,y)\leq\sqrt{2}p_G(x,y)$. Here, the equality holds for $G=\R^2\backslash\{0\}$, $x=-1$ and $y=1$, since now $s_G(x,y)=\sqrt{2}p_G(x,y)=1$.
\end{proof}

\begin{remark}
Theorem \ref{sp_commonbounds} improves the upper bound of \cite[Lemma 2.5(1), p. 1126]{hvz}. 
\end{remark}

\begin{theorem}\label{sp_convex}
A domain $G\subsetneq\R^n$ is convex if and only if the inequality $s_G(x,y)\leq p_G(x,y)$ holds for all $x,y\in G$.
\end{theorem}
\begin{proof}
By \cite[Lemma 11.6(1), p. 197]{hkvbook}, $s_G(x,y)\leq p_G(x,y)$ holds for all $x,y\in G$ if $G$ is convex. Suppose now that $G$ is non-convex and fix $x,y\in G$ so that $[x,y]\cap\partial G\neq\varnothing$. Clearly, $s_G(x,y)\geq p_G(x,y)$ because 
\begin{align*}
\inf_{z\in\partial G}(|x-z|+|z-y|)=|x-y|<\sqrt{|x-y|^2+4d_G(x)d_G(y)}.   
\end{align*}
Consequently, if $s_G(x,y)\leq p_G(x,y)$ holds for all $x,y\in G$, then $G$ must be convex.
\end{proof}

\begin{theorem}\label{thm_sp2}
For a fixed angle $\theta\in(0,\pi)$ and for all $x,y\in S_\theta$, the sharp inequality 
\begin{align*}
\frac{p_{S_\theta}(x,y)}{\sqrt{2}\cos(\theta\slash4)}\leq s_{S_\theta}(x,y)    
\end{align*}
holds.
\end{theorem}
\begin{proof}
Consider the quotient
\begin{align}\label{sp_quo}
\frac{s_{S_\theta}(x,y)}{p_{S_\theta}(x,y)}=\frac{\sqrt{|x-y|^2+4d_{S_\theta}(x)d_{S_\theta}(y)}}{\inf_{z\in\partial S_\theta}(|x-z|+|z-y|)}.   
\end{align}
Suppose without loss of generality that $x=e^{hi}$ and $y=re^{ki}$ with $0<h\leq k<\theta$ and $r>0$. Assume also that $h<\theta\slash2<k$, since otherwise $p_{S_\theta}(x,y)=s_{S_\theta}(x,y)$. Now, the infimum in the quotient \eqref{sp_quo} is $\min\{|\overline{x}-y|,|x-y'|\}$, where $\overline{x}$ is the complex conjugate of $x$ and $y'$ is the point $y$ reflected over the other side of the sector. Clearly, $|\overline{x}-y|=|1-re^{(k+h)i}|$ and $|x-y'|=|1-re^{(2\theta-k-h)i}|$, so, to ensure the quotient is at minimum, we need to fix
\begin{align*}
|1-re^{(k+h)i}|=|1-re^{(2\theta-k-h)i}|\quad\Leftrightarrow\quad k+h=2\theta-k-h\quad\Leftrightarrow\quad k=\theta-h.   
\end{align*}
It follows that the quotient \eqref{sp_quo} is now
\begin{align*}
\sqrt{\frac{1+r^2-2r\cos(\theta-2h)+4r\sin^2(h)}{1+r^2-2r\cos(\theta)}}.    
\end{align*}
Regardless of the exact values of $r$ or $\theta$, it can be shown with differentiation that the expression $1+r^2-2r\cos(\theta-2h)+4r\sin^2(h)$ attains its minimum value when $h=\theta\slash4$. Thus, the above quotient is minorized by
\begin{align*}
\sqrt{\frac{1+2r+r^2-4r\cos(\theta\slash2)}{1+r^2-2r\cos(\theta)}}. \end{align*}
By using differentiation again, we can show that the expression above attains its minimum value with $r=1$. Thus, the minimum value of the quotient \eqref{sp_quo} is
\begin{align*}
\sqrt{\frac{2-2\cos(\theta\slash2)}{1-\cos(\theta)}}=\frac{1}{\sqrt{2}\cos(\theta\slash4)}.    
\end{align*}
\end{proof}

\begin{theorem}\label{thm_sp3}
For a fixed angle $\theta\in[\pi,2\pi)$ and for all $x,y\in S_\theta$, the sharp inequality $p_{S_\theta}(x,y)\leq s_{S_\theta}(x,y)$ holds.
\end{theorem}
\begin{proof}
Just like the proof of Theorem \ref{thm_sp2}, fix $x=e^{hi}$ and $y=re^{ki}$ with $0<h<\theta\slash2<k<\theta$ and $r>0$. Suppose also that $k-h<\pi$, for otherwise $[x,y]\cap\partial S_\theta\neq\varnothing$ and $s_{S_\theta}(x,y)=1\geq p_{S_\theta}(x,y)$. Denote $q=\inf_{z\in \partial S_\theta}(|x-z|+|z-y|)$. If $k+h\leq\pi$, then $q=|\overline{x}-y|$, where $\overline{x}=e^{-hi}$ is the point $x$ reflected over the real axis. If $k+h>\pi$, then $k+h>2\theta-\pi$ and $q=|x-y'|$, where $y'=re^{(2\theta-k)i}$ is the point $y$ reflected over the other side of the sector. Thus, by symmetry, we can set $k+h\leq\pi$. Now, $p_{S_\theta}(x,y)\leq s_{S_\theta}(x,y)$ because, for $\theta\in[\pi,2\pi)$,
\begin{align*}
q=|\overline{x}-y|=\sqrt{|x-y|^2+4d_{\uhp^2}(x)d_{\uhp^2}(y)}
\leq\sqrt{|x-y|^2+4d_{S_\theta}(x)d_{S_\theta}(y)}.
\end{align*}
\end{proof} 

\begin{theorem}\label{thm_sp4}
For a fixed angle $\theta\in(\pi,2\pi)$ and for all $x,y\in S_\theta$, the sharp inequality 
\begin{align*}
s_{S_\theta}(x,y)\leq\sqrt{2}\sin(\theta\slash4)
p_{S_\theta}(x,y)    
\end{align*}
holds.
\end{theorem}
\begin{proof}
We are now interested in the maximum value of the quotient \eqref{sp_quo} for $\theta\in(\pi,2\pi)$. Just like in the proof of Theorem \ref{thm_sj3}, fix $x=e^{hi}$ and $y=e^{(\theta-h)i}$ with $(\theta-\pi)\slash2<h<\theta\slash2$. It follows that the quotient \eqref{sp_quo} is either 
\begin{align*}
&\sqrt{\sin^2(\theta\slash2-h)+\sin^2(h)}\text{ with }(\theta-\pi)\slash2<h<\pi\slash2,\text{ or }\\
&\sqrt{\sin^2(\theta\slash2-h)+1}\text{ with }\pi\slash2\leq h<\theta\slash2.
\end{align*}
In the first case, the maximum value of the quotient \eqref{sp_quo} is $\sqrt{2}\sin(\theta\slash4)$ with $h=\theta\slash2$ and, in the second case, the maximum is $\sin(\theta\slash2)$ with $h=\pi\slash2$. Since $\sqrt{2}\sin(\theta\slash4)>\sin(\theta\slash2)$ for all $\theta\in(\pi,2\pi)$, the greatest value of the quotient \eqref{sp_quo} is $\sqrt{2}\sin(\theta\slash4)$.
\end{proof}

\begin{theorem}\label{fin_thm_sjp_sector}
For a fixed angle $\theta\in(0,2\pi)$, the following inequalities hold:\\
(1) $j^*_{S_\theta}(x,y)\leq p_{S_\theta}(x,y)\leq\sqrt{2}j^*_{S_\theta}(x,y)$ if $\theta\in(0,2\pi)$,\\
(2) $j^*_{S_\theta}(x,y)\leq s_{S_\theta}(x,y)\leq\sqrt{2}j^*_{S_\theta}(x,y)$ if $\theta\in(0,\pi]$,\\
(3) $j^*_{S_\theta}(x,y)\leq s_{S_\theta}(x,y)\leq2\sin(\theta\slash4)j^*_{S_\theta}(x,y)$ if $\theta\in(\pi,2\pi)$,\\
(4) $(\sqrt{2}\cos(\theta\slash4))^{-1}p_{S_\theta}(x,y)\leq s_{S_\theta}(x,y)\leq
p_{S_\theta}(x,y)$ if $\theta\in(0,\pi]$,\\
(5) $p_{S_\theta}(x,y)\leq s_{S_\theta}(x,y)\leq
\sqrt{2}\sin(\theta\slash4)p_{S_\theta}(x,y)$ if $\theta\in(\pi,2\pi)$.\\
Furthermore, the constants are sharp in each case.
\end{theorem}
\begin{proof}
The inequality (1) and its sharpness follow from Theorem \ref{jp_bounds} and Lemmas \ref{sjp_equ_possible} and \ref{jp_sharproot}. The inequality (2) holds by Theorems \ref{sj_convexbounds} and its sharpness follows from Lemma \ref{sjp_equ_possible} and the fact that, for $k=\sin(\min\{\theta\slash2,\pi\slash4\})$, $x=1+ki$ and $y=1+2k+ki$, the equality $s_{S_\theta}(x,y)=\sqrt{2}j^*_{S_\theta}(x,y)$ holds. By Theorem \ref{sj_properbounds}, Lemma \ref{sjp_equ_possible} and Theorem \ref{thm_sj3}, the inequality (3) holds and is sharp. The inequality (4) and its sharpness follow from Theorem \ref{sp_convex}, Lemma \ref{sjp_equ_possible} and Theorem \ref{thm_sp2}. Finally, the inequality (5) holds and is sharp by Theorems \ref{thm_sp3} and \ref{thm_sp4}.
\end{proof}

\begin{theorem}\label{strip_ine_spj}
The following inequalities hold for all $x,y\in S_0$:\\
(1) $j^*_{S_0}(x,y)\leq p_{S_0}(x,y)\leq\sqrt{2}j^*_{S_0}(x,y)$,\\
(2) $j^*_{S_0}(x,y)\leq s_{S_0}(x,y)\leq\sqrt{2}j^*_{S_0}(x,y)$,\\
(3) $p_{S_0}(x,y)\slash\sqrt{2}\leq s_{S_0}(x,y)\leq p_{S_0}(x,y)$.\\
Furthermore, in each case the constants are sharp.
\end{theorem}
\begin{proof}
The inequality (1) and its sharpness follow from Theorem \ref{jp_bounds} and Lemmas \ref{sjp_equ_possible} and \ref{jp_sharproot}. The inequalities (2) and (3) hold by Theorems \ref{sj_convexbounds}, \ref{sp_commonbounds} and \ref{sp_convex}. They are sharp, too: For $x=1+i$ and $y=3+i$, $s_{S_0}(x,y)=p_{S_0}(x,y)=\sqrt{2}j^*_{S_0}(x,y)$ and, for $x=(\pi\slash4)i$ and $y=(3\pi\slash4)i$, $s_{S_0}(x,y)=j^*_{S_0}(x,y)=p_{S_0}(x,y)\slash\sqrt{2}$.
\end{proof}
\section{Hyperbolic Metric in a Sector}\label{rhos_section}
The main result of this section is Corollary \ref{cor_rhos} which compares the triangular ratio
metric and the hyperbolic metric of a sector domain. To prove it, we construct
a conformal self-map of the sector, mapping two points in a general position to
a pair of points, symmetric with respect to the bisector of the sector angle.
Because conformal maps preserve the hyperbolic distance, under this mapping
the hyperbolic distance remains invariant whereas the triangular ratio distance
may change. This enables us to reduce the comparison of these two metrics to the case
when the points are symmetric with respect to the bisector.

\begin{proposition}\label{prop_orthocircles}
Let $x,y\in\uhp^2$ be two distinct points, let $L(x,y)$ be the line through them, and let the angle of intersection between $L(x,y)$  and the real axis be $\alpha$ and suppose that $\alpha\in(0,\pi\slash2)$. Then there are two circles $S^1(c_1,r_1)$ and $S^1(c_2,r_2)$, centered at the real axis and orthogonal to each other, such that $x,y\in S^1(c_1,r_1)$ and $c_2=L(x,y)\cap\R$. 
\end{proposition}
\begin{proof}
First, fix $\beta=\pi\slash2+\arg(x-y)$ so that $[x,y]\perp[0,e^{\beta i}]$. Choose now
\begin{align*}
c_1=L(0,1)\cap L\left(\frac{x+y}{2},\frac{x+y}{2}+e^{\beta i}\right),\quad
c_2=L(0,1)\cap L(x,y).
\end{align*}
Let yet $r_2$ be such that $r_1^2+r_2^2=(c_1-c_2)^2$, so that the two circles are orthogonal.
\end{proof}

\begin{lemma}\label{lemma_gmobius}
For given two distinct points $x\,, y \in \uhp^2$, there exists a M\"obius transformation
$g:\uhp^2\to\uhp^2$ such that $|g(x)|=|g(y)| =1$ and ${\rm Im}(g(x)) = {\rm Im}(g(y))\,.$

(1) If ${\rm Im}(x) = {\rm Im}(y)\,,$ then $g(z)=(z-a)/r$
where $a=  {\rm Re}((x+y)/2)$ and $r =|x-a|\,.$

(2) If ${\rm Re}(x) = {\rm Re}(y) =a\,$ and $r=\sqrt{{\rm Im}(x)  {\rm Im}(y)}$, then $g$ is the M\"obius transformation fulfilling $g(a-r)=0$, $g(a)=1$ and $g(a+r) =\infty$.

(3) In the remaining case, the angle $\alpha=\measuredangle(L(x,y),\R)$ belongs to $(0,\pi\slash2)$. Let $S^1(c_1,r_1)$ and $S^1(c_2,r_2)$ be as in Proposition \ref{prop_orthocircles}. Then $g$ is determined by $g(B^2(c_1,r_1)\cap\uhp^2)=\B^2\cap\uhp^2\,,$  $g(c_1-r_1)=-1$, $g(c_1+r_1)=1$ and $g(S^1(c_2,r_2)\cap\uhp^2)=\{yi\text{ }|\text{ }y>0\}$.
\end{lemma}
\begin{proof} (1) This case is obvious.

(2) Since $g(S^1(a,r))$ is the imaginary axis, $g(\{z\text{ }|\text{ }{\rm Re}(z)=a\})=S^1(0,1)$ and, because
$S^1(a,r)$ passes through the hyperbolic midpoint of the segment $J[x,y]$ \cite[pp. 52-53]{hkvbook}, it
follows that the required conditions hold.

(3) There are two possible cases. If $c_1-r_1<c_2-r_2<c_1+r_1$, the transformation $g$ fulfills $g(u)=0$ for $u=c_2-r_2$ and is given by 
\begin{align}\label{g_formula}
g(z)=\frac{r_1(z+r_2-c_2)}{(c_1+r_2-c_2)z+r_1^2+c_1c_2-c_1^2-c_1r_2}.  \end{align}
See Figure \ref{fig1}. Otherwise $c_1-r_1<c_2+r_2<c_1+r_1$, $g$ is given by the formula \eqref{g_formula}, substituting $r_2$ by $-r_2$, and it holds that $g(c_2+r_2)=0$.
\end{proof}

\begin{figure}
    \centering
    \begin{tikzpicture}[scale=1.3]
    \draw[thick] (0,0) -- (6.133,0);
    \draw[thick] (0,3) -- (4.6,0);
    \draw[thick] (4.2,0) arc (0:180:2.1cm);
    \draw[thick] (6.133,0) arc (0:180:1.533cm);
    \draw (2,0) circle (0.05cm);
    \draw (4.6,0) circle (0.05cm);
    \draw (1.51,2.02) circle (0.05cm);
    \draw (4.18,0.27) circle (0.05cm);
    \draw (3.067,0) circle (0.05cm);
    \node[scale=1.3] at (1.51,2.3) {$x$};
    \node[scale=1.3] at (4.3,0.5) {$y$};
    \node[scale=1.3] at (2,0.3) {$c_1$};
    \node[scale=1.3] at (4.8,0.2) {$c_2$};
    \node[scale=1.3] at (1,3) {$L(x,y)$};
    \node[scale=1.3] at (3.067,-0.3) {$u$};
    \draw[thick] (7,0) -- (10,0);
    \draw[thick] (8.5,0) -- (8.5,3.3);
    \draw[thick] (9.9,0) arc (0:180:1.4cm);
    \draw (7.7,1.15) circle (0.05cm);
    \draw (9.3,1.15) circle (0.05cm);
    \draw (8.5,0) circle (0.05cm);
    \node[scale=1.3] at (7.5,1.6) {$g(x)$};
    \node[scale=1.3] at (9.5,1.6) {$g(y)$};
    \node[scale=1.3] at (7,-0.3) {$-1$};
    \node[scale=1.3] at (10,-0.3) {$1$};
    \node[scale=1.3] at (8.65,1.7) {$i$};
    \node[scale=1.3] at (8.5,-0.3) {$g(u)$};
    \end{tikzpicture}
    \caption{Circles $S^1(c_1,r_1)$ and $S^1(c_2,r_2)$ before and after the transformation $g$ of Lemma \ref{lemma_gmobius}.}
    \label{fig1}
\end{figure}
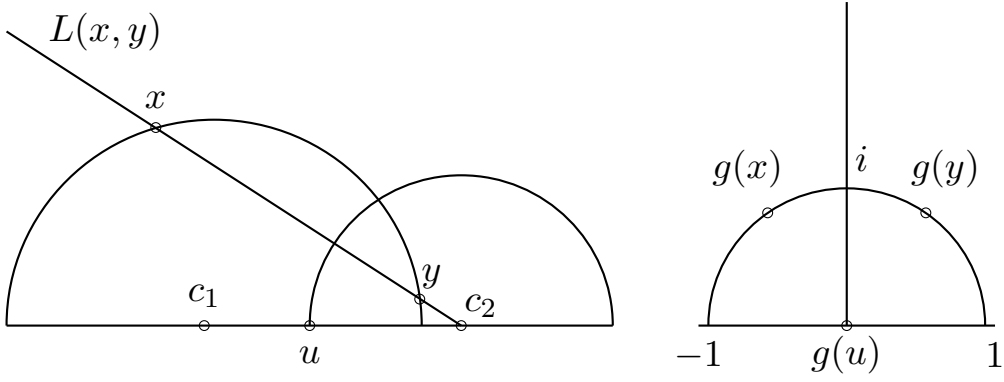

\begin{remark}
The proof of Lemma \ref{lemma_gmobius} shows that the circle $S^1(c_2,r_2)$ bisects the hyperbolic segment $J[x,y]$ joining $x$ and $y$ in $\uhp^2$.
\end{remark}

\begin{lemma}\label{lemma_simplifying_rhosquo}
For all $\theta\in(0,2\pi)$ and $x,y\in S_\theta$, there exists a conformal mapping $f:S_\theta\to S_\theta$ such that $f(x)=e^{(1-k)\theta i\slash2}$ and $f(y)=e^{(1+k)\theta i\slash2}$ for some $k\in(0,1)$. Furthermore,
\begin{align*}
\inf_{0<k<1}Q(f(x),f(y))\leq Q(x,y)\equiv\frac{{\rm th}(\rho_{S_\theta}(x,y)\slash2)}{s_{S_\theta}(x,y)}\leq\sup_{0<k<1}Q(f(x),f(y)).    
\end{align*}
\end{lemma}
\begin{proof}
Let $h:S_\theta\to\uhp^2$ be the conformal map $h(z)=z^{\pi\slash\theta}$ for all $z\in S_\theta$. By Lemma \ref{lemma_gmobius}, we find a M\"obius transformation $g:\uhp^2\to\uhp^2$ such that $|g(h(x))|=|g(h(y))|=1$ and Im$(g(h(x)))=$Im$(g(h(y)))$. It follows that the points $h^{-1}(g(h(x)))$ and $h^{-1}(g(h(y)))$ are $e^{(1-k)\theta i\slash2}$ and $e^{(1+k)\theta i\slash2}$, respectively, for some $k\in(0,1)$. Thus, we can always form a suitable conformal map $f=h^{-1}\circ g\circ h$, and the latter part of our theorem follows from the conformal invariance of the hyperbolic metric.
\end{proof}

\begin{theorem}\label{rhos_xy1_underpi}
For a fixed angle $\theta\in(0,\pi)$ and for all $x,y\in S_\theta$, the sharp inequality 
$s_{S_\theta}(x,y)\leq
{\rm th}(\rho_{S_\theta}(x,y)\slash2)\leq
(\pi\slash\theta)\sin(\theta\slash2) s_{S_\theta}(x,y)$ holds.
\end{theorem}
\begin{proof}
Consider the quotient
\begin{align}\label{rhos_quo}
\frac{\text{th}(\rho_{S_\theta}(x,y)\slash2)}{s_{S_\theta}(x,y)}=\frac{|x^{\pi\slash\theta}-y^{\pi\slash\theta}|\inf_{z\in\partial S_\theta}(|x-z|+|z-y|)}{|x^{\pi\slash\theta}-\overline{(y^{\pi\slash\theta})}||x-y|}.
\end{align}
By Lemma \ref{lemma_simplifying_rhosquo}, we can choose points so that $x=e^{(1-k)\theta i\slash2}$ and $y=e^{(1+k)\theta i\slash2}$, where $k\in(0,1)$ without loss of generality. It follows that the quotient \eqref{rhos_quo} is
\begin{align*}
\frac{\sin(k\pi\slash2)\sin(\theta\slash2)}{\sin(k\theta\slash2)},
\end{align*}
which is decreasing with respect to $k$. Thus, the extreme values of the quotient are
\begin{align*}
\lim_{k\to1^-}
\left(\frac{\sin(k\pi\slash2)\sin(\theta\slash2)}{\sin(k\theta\slash2)}\right)=1
\quad\text{and}\quad
\lim_{k\to0^+}
\left(\frac{\sin(k\pi\slash2)\sin(\theta\slash2)}{\sin(k\theta\slash2)}\right)=\frac{\pi}{\theta}\sin\left(\frac{\theta}{2}\right).
\end{align*}
\end{proof}

\begin{theorem}\label{rhos_xy1_overpi}
For a fixed angle $\theta\in(\pi,2\pi)$ and for all $x,y\in S_\theta$, the sharp inequality
\begin{align*}
(\pi\slash\theta)s_{S_\theta}(x,y)\leq
{\rm th}(\rho_{S_\theta}(x,y)\slash2)\leq
 s_{S_\theta}(x,y)    
\end{align*}
holds.
\end{theorem}
\begin{proof}
Just like in the proof of Theorem \ref{rhos_xy1_underpi}, we can fix $x=e^{(1-k)\theta i\slash2}$ and $y=e^{(1+k)\theta i\slash2}$. The quotient \eqref{rhos_quo} is
\begin{align*}
\frac{\text{th}(\rho_{S_\theta}(x,y)\slash2)}{s_{S_\theta}(x,y)}=\frac{\sin(k\pi\slash2)}{\sin(k\theta\slash2)}\quad\text{or}\quad
\sin\left(\frac{k\pi}{2}\right),
\end{align*}
depending on if $k<\pi\slash\theta$ or not. It has a minimum value 
\begin{align*}
\lim_{k\to0^+}
\left(\frac{\sin(k\pi\slash2)}{\sin(k\theta\slash2)}\right)=\frac{\pi}{\theta},
\end{align*} 
and a maximum value $\lim_{k\to1^-}\sin(k\pi\slash2)=1$.
\end{proof}

\begin{corollary}\label{cor_rhos}
For a fixed angle $\theta\in(0,2\pi)$ and for all $x,y\in S_\theta$, the following results hold:\newline 
(1) $s_{S_\theta}(x,y)\leq
{\rm th}(\rho_{S_\theta}(x,y)\slash2)\leq
(\pi\slash\theta)\sin(\theta\slash2) s_{S_\theta}(x,y)$ if $\theta\in(0,\pi)$,\newline
(2) $s_{S_\theta}(x,y)={\rm th}(\rho_{S_\theta}(x,y)\slash2)$ if $\theta=\pi$,\newline
(3) $(\pi\slash\theta)s_{S_\theta}(x,y)\leq
{\rm th}(\rho_{S_\theta}(x,y)\slash2)\leq
 s_{S_\theta}(x,y)$ if $\theta\in(\pi,2\pi)$.\newline
 Furthermore, these bounds are also sharp.
\end{corollary}
\begin{proof}
Follows from Theorems \ref{rhos_xy1_underpi} and \ref{rhos_xy1_overpi}, and from \cite[Rmk 2.9 p. 687]{chkv}.
\end{proof}

\begin{corollary}\label{cor_rhoj}
For a fixed angle $\theta\in(0,2\pi)$ and for all $x,y\in S_\theta$, the following results hold:\newline 
(1) $j^*_{S_\theta}(x,y)\leq
{\rm th}(\rho_{S_\theta}(x,y)\slash2)\leq
\sqrt{2}(\pi\slash\theta)\sin(\theta\slash2) j^*_{S_\theta}(x,y)$ if $\theta\in(0,\pi)$,\newline
(2) $j^*_{S_\theta}(x,y)\leq
{\rm th}(\rho_{S_\theta}(x,y)\slash2)\leq\sqrt{2}j^*_{S_\theta}(x,y)$ if $\theta=\pi$,\newline
(3) $(\pi\slash\theta)j^*_{S_\theta}(x,y)\leq
{\rm th}(\rho_{S_\theta}(x,y)\slash2)\leq2\sin(\theta\slash4) j^*_{S_\theta}(x,y)$ if $\theta\in(\pi,2\pi)$.
\end{corollary}
\begin{proof}
Follows from Theorem \ref{fin_thm_sjp_sector} and Corollary \ref{cor_rhos}.
\end{proof}

\begin{corollary}\label{cor_rhop}
For a fixed angle $\theta\in(0,2\pi)$ and for all $x,y\in S_\theta$, the following results hold:\newline 
(1) $p_{S_\theta}(x,y)\slash(\sqrt{2}\cos(\theta\slash4))\leq
{\rm th}(\rho_{S_\theta}(x,y)\slash2)\leq(\pi\slash\theta)\sin(\theta\slash2) p_{S_\theta}(x,y)$ if $\theta\in(0,\pi)$,\newline
(2) $p_{S_\theta}(x,y)={\rm th}(\rho_{S_\theta}(x,y)\slash2)$ if $\theta=\pi$,\newline
(3) $(\pi\slash\theta)p_{S_\theta}(x,y)\leq
{\rm th}(\rho_{S_\theta}(x,y)\slash2)\leq\sqrt{2}\sin(\theta\slash4)
 p_{S_\theta}(x,y)$ if $\theta\in(\pi,2\pi)$.
\end{corollary}
\begin{proof}
Follows from Theorem \ref{fin_thm_sjp_sector}, Corollary \ref{cor_rhos} and \cite[Rmk 2.9 p. 687]{chkv}.
\end{proof}

\begin{theorem}\label{strip_rhos}
For all $x,y\in S_0$, the sharp inequality
\begin{align*}
s_{S_0}(x,y)\leq
{\rm th}(\rho_{S_0}(x,y)\slash2)\leq
(\pi\slash2)s_{S_0}(x,y)   
\end{align*}
holds.
\end{theorem}
\begin{proof}
First, note that there is a conformal mapping $h:S_0\to\uhp^2$, $h(z)=e^z$. By using the M\"obius transformation $g$ of Lemma \ref{lemma_gmobius}, we can create a conformal mapping $f=h^{-1}\circ g\circ h:S_0\to S_0$ such that $f(x)=(1-k)\pi i\slash2$ and $f(y)=(1+k)\pi i\slash2$ for some $k\in(0,1)$. Just like in the proof of Lemma \ref{lemma_simplifying_rhosquo}, it follows that 
\begin{align*}
\inf_{0<k<1}Q(f(x),f(y))\leq Q(x,y)\equiv\frac{\text{th}(\rho_{S_0}(x,y)\slash2)}{s_{S_0}(x,y)}\leq\sup_{0<k<1}Q(f(x),f(y)).    
\end{align*}
Consider now the quotient
\begin{align}\label{rhos_quo_strip}
Q(f(x),f(y))=\frac{\text{th}(\rho_{S_0}(f(x),f(y))\slash2)}{s_{S_0}(f(x),f(y))}.    
\end{align}
Clearly, $s_{S_0}(f(x),f(y))=s_{S_0}((1-k)\pi i\slash2,(1+k)\pi i\slash2)=k$. Furthermore,
\begin{align*}
\text{th}\frac{\rho_{S_0}(f(x),f(y))}{2}=
\text{th}\frac{\rho_{\uhp^2}(h(f(x)),h(f(y)))}{2}=
\sin(k\pi\slash2).
\end{align*}
The quotient \eqref{rhos_quo_strip} is therefore $\sin(k\pi\slash2)\slash k$. By differentiation, it can be shown that this result is decreasing with regards to $k$, so its minimum value is $\lim_{k\to1^-}(\sin(k\pi\slash2)\slash k)=1$ and its maximum value $\lim_{k\to0^+}(\sin(k\pi\slash2)\slash k)=\pi\slash2$.
\end{proof}

\begin{corollary}\label{cor_strip_rhojp}
For all $x,y\in S_0$, the following inequalities hold:\\
(1) $j^*_{S_0}(x,y)\leq
{\rm th}(\rho_{S_0}(x,y)\slash2)\leq
(\pi\slash\sqrt{2})j^*_{S_0}(x,y)$,\\
(2) $p_{S_0}(x,y)\slash\sqrt{2}\leq
{\rm th}(\rho_{S_0}(x,y)\slash2)\leq
(\pi\slash2)p_{S_0}(x,y)$.
\end{corollary}
\begin{proof}
Follows from Theorems \ref{strip_ine_spj}(2), \ref{strip_ine_spj}(3) and \ref{strip_rhos}.
\end{proof}

\begin{remark}
The inequalities of Theorem \ref{strip_rhos} and Corollary \ref{cor_strip_rhojp} are the same as the inequalities of Corollaries \ref{cor_rhos}, \ref{cor_rhoj} and \ref{cor_rhop} when $\theta\to0^+$.
\end{remark}
\section{s-Metric in Quasiconformal Mappings}\label{s_in_qms_section}
The main result of this section and the whole  paper is Corollary \ref{cor_quasi_s_bothways}. First, we introduce a general result related to the triangular ratio metric under quasiconformal mappings and then we develop it further with the inequalities of Corollary \ref{cor_rhos}. At the end of this section, we also consider the triangular ratio metric under a conformal mapping between sectors. The behaviour of the triangular ratio metric under M\"obius transformations and quasiconformal mappings has been studied earlier; see \cite[Thms 1.2 \& 1.3 p. 684; Cor. 3.30 \& Thm 3.31 p. 697]{chkv}, \cite[Thm 4.7 p. 1144; Thm 4.9 p. 1146]{hvz}.

For the definition and basic properties of $K$-quasiconformal homeomorphisms, the reader
is referred to \cite[Ch.2]{v1}.
We start with two preliminary results.

\begin{lemma}\label{lemma_ine_dt}
For all $t\in(0,1)$ and $d\geq1$, the inequality ${\rm th}(d{\rm arth}(t))\leq dt$ holds.
\end{lemma}
\begin{proof}
For all $t>0$, $\text{th}(t)\slash t$ is decreasing. This is because, by differentiation,
\begin{align*}
\frac{d}{dt}\left(\frac{\text{th}(t)}{t}\right)=\frac{1}{t\text{ch}^2(t)}-\frac{\text{th}(t)}{t^2}
=\frac{1}{t\text{ch}(t)}\left(\frac{1}{\text{ch}(t)}-\frac{\text{sh}(t)}{t}\right)\leq0\quad\Leftrightarrow\quad t\leq\text{sh}(t)\text{ch}(t),
\end{align*}
which clearly holds since $\text{ch}(t)\geq\text{sh}(t)\geq t$ for $t>0$. It follows from this that
\begin{align*}
\frac{\text{th}(d\text{arth}(t))}{d\text{arth}(t)}\leq\frac{\text{th}(\text{arth}(t))}{\text{arth}(t)}=
\frac{t}{\text{arth}(t)}
\quad\Leftrightarrow\quad
\text{th}(d\text{arth}(t))\leq
dt.
\end{align*}
\end{proof}

\begin{lemma}\label{thm_w_final}
For all $t\in(0,1)$ and $c\geq K\geq1$,
\begin{align*}
w(t)\leq\max\{1,d^{1\slash K}\}t^{1\slash K},    \quad w(t)\equiv{\rm th}\left(\frac{c}{2}(2{\rm arth}(t))^{1\slash K}\right),
\end{align*}
where $d=2(c \slash2)^K$.
\end{lemma}
\begin{proof}
The function $h(\alpha)=\text{th}(u^\alpha)^{1\slash\alpha}$, where $u>0$ is fixed, is increasing for $\alpha>0$ because, by differentiation,
\begin{align*}
\frac{d}{d\alpha}\log(h(\alpha))&=
\frac{u^\alpha\log(u)}{\alpha^2\text{th}(u^\alpha)\text{ch}^2(u^\alpha)}
-\frac{\log(\text{th}(u^\alpha))}{\alpha^3}
-\frac{\log(\text{th}(u^\alpha)^{1\slash\alpha})}{\alpha^2}\\
&=\frac{u^\alpha\log(u))}{\alpha\text{th}(u^\alpha)\text{ch}^2(u^\alpha)}
-\frac{\log(\text{th}(u^\alpha))}{\alpha^2}>0.
\end{align*}
It follows that
\begin{align*}
w(t)^K=\text{th}((d\text{arth}(t))^{1\slash K})^K\leq\text{th}(d\text{arth}(t))
\quad\Rightarrow\quad w(t)\leq\text{th}(d\text{arth}(t))^{1\slash K}\equiv A.
\end{align*}
Clearly, $A\leq\text{th}(\text{arth}(t))^{1\slash K}=t^{1\slash K}$ if $d\in(0,1]$. If $d>1$ instead, by Lemma \ref{lemma_ine_dt}, $A\leq d^{1\slash K}t^{1\slash K}$.
\end{proof}

\begin{theorem}\label{thm_holderH}
The function
\begin{align*}
H(t)={\rm th}\left(\frac{C}{2}\max\{2{\rm arth}(t),(2{\rm arth}(t))^{1\slash K}\}\right),    
\end{align*}
where $t\in(0,1)$, $K\geq1$ and $C\geq1$, fulfills $H(t)\leq Ct^{1\slash K}$ and is thus H\"older continuous with exponent $1\slash K$.
\end{theorem}
\begin{proof}
Observe first that
\begin{align*}
2\text{arth}(t)\geq(2\text{arth}(t))^{1\slash K}
\quad\Leftrightarrow\quad 2\text{arth}(t)\geq1
\quad\Leftrightarrow\quad t\geq\text{th}\left(\frac{1}{2}\right)=\frac{e-1}{e+1}\equiv t_1.
\end{align*}
If $t\in(0,t_1)$, by Lemma \ref{thm_w_final},
\begin{align*}
H(t)&=\text{th}\left(\frac{C}{2}(2\text{arth}(t))^{1\slash K}\right)
\leq\max\{1,\frac{C}{2^{1-1\slash K}}\}t^{1\slash K}
\leq Ct^{1\slash K}.
\end{align*}
If $t\in[t_1,1)$ instead, by Lemma \ref{lemma_ine_dt},
\begin{align*}
H(t)=\text{th}(C\text{arth}(t))\leq Ct\leq Ct^{1\slash K}. 
\end{align*}
Thus, the inequality $H(t)\leq Ct^{1\slash K}$ holds for all $t\in(0,1)$.
\end{proof}

The main results of this section are based on the following recent form of the quasiconformal Schwarz lemma:

\begin{theorem}\label{hol_1.1}
Let $G_1$ and $G_2$ be simply-connected domains in $\R^2$ and $f:G_1\to G_2=f(G_1)$ a $K-$quasiconformal homeomorphism. Then
\begin{align*}
\rho_{G_2}(f(x),f(y))\leq c(K)\max\{\rho_{G_1}(x,y),\rho_{G_1}(x,y)^{1\slash K}\} \end{align*}
for all $x,y\in G_1$ where $c(K)$ is as in \emph{\cite[Thm 16.39, p. 313]{hkvbook}, \cite[Theorem 3.6]{wv}}.
\end{theorem}

\begin{remark}
By \cite[Thm 16.39, p. 313]{hkvbook}, 
\begin{align*}
K\leq c(K)\leq\log(2(1+\sqrt{1-1\slash e^2}))(K-1)+K
\end{align*}
and, in particular, $c(K)\to1$, when $K\to1$.
\end{remark}

\begin{corollary}\label{s_quasi_ine}
Let $G_1$ and $G_2$ be simply-connected domains in $\R^2$ and $f:G_1\to G_2=f(G_1)$ a $K$-quasiconformal homeomorphism. Suppose that there exist $A,B\in(0,\infty)$ so that $As_{G_2}(u,v)\leq{\rm th}(\rho_{G_2}(u,v)\slash 2)$ for all $u,v\in G_2$ and ${\rm th}(\rho_{G_1}(x,y)\slash 2)\leq Bs_{G_1}(x,y)$ for all $x,y\in G_1$. Then, for all $x,y\in G_1$,
\begin{align*}
s_{G_2}(f(x),f(y))\leq\frac{c(K)B^{1\slash K}}{A}s_{G_1}(x,y)^{1\slash K}. 
\end{align*}
\end{corollary}
\begin{proof}
By Theorem \ref{hol_1.1},
\begin{align*}
&2\text{arth}(s_{G_2}(f(x),f(y)))
\leq 2\text{arth}\left(\frac{1}{A}\text{th}\left(\frac{\rho_{G_2}(f(x),f(y))}{2}\right)\right)\\
&\leq2\text{arth}\left(\frac{1}{A}\text{th}\left(\frac{c(K)}{2}\max\{\rho_{G_1}(x,y),\rho_{G_1}(x,y)^{1\slash K}\}\right)\right)\\
&\leq2\text{arth}\left(\frac{1}{A}\text{th}\left(\frac{c(K)}{2}\max\{2\text{arth}(Bs_{G_1}(x,y)),(2\text{arth}(Bs_{G_1}(x,y)))^{1\slash K}\}\right)\right),
\end{align*}
and, applying  Theorem \ref{thm_holderH} with $C=c(K)$, we will have
\begin{align*}
s_{G_2}(f(x),f(y))\leq\frac{1}{A}H(Bs_{G_1}(x,y))\leq \frac{c(K)B^{1/K}}{A}s_{G_1}(x,y)^{1/K}.
\end{align*}
\end{proof}

\begin{corollary}\label{cor_quasi_s_bothways}
If $\alpha,\beta\in(0,2\pi)$ and $f:S_\alpha\to S_\beta=f(S_\alpha)$ is a $K$-quasiconformal homeomorphism, the following inequalities hold for all $x,y\in S_\alpha$.
\begin{align*}
&(1) \quad 
\frac{\beta}{c(K)^K\pi\sin(\beta\slash2)}s_{S_\alpha}(x,y)^K\leq s_{S_\beta}(f(x),f(y))\leq c(K)\left(\frac{\pi}{\alpha}\sin\left(\frac{\alpha}{2}\right)\right)^{1\slash K}s_{S_\alpha}(x,y)^{1\slash K}\\
&\quad\quad\text{if }\alpha,\beta\in(0,\pi],\\
&(2) \quad
\frac{1}{c(K)^K}s_{S_\alpha}(x,y)^K\leq s_{S_\beta}(f(x),f(y))\leq \frac{c(K)\pi}{\beta}\left(\frac{\pi}{\alpha}\sin\left(\frac{\alpha}{2}\right)\right)^{1\slash K}s_{S_\alpha}(x,y)^{1\slash K}\\
&\quad\quad\text{if }\alpha\in(0,\pi)\text{ and }\beta\in(\pi,2\pi),\\
&(3)\quad
\left(\frac{\alpha}{c(K)\pi}\right)^Ks_{S_\alpha}(x,y)^K\leq s_{S_\beta}(f(x),f(y))\leq \frac{c(K)\pi}{\beta}s_{S_\alpha}(x,y)^{1\slash K}\\
&\quad\quad\text{if }\alpha,\beta\in[\pi,2\pi).
\end{align*}
\end{corollary}
\begin{proof}
Follows from Corollaries \ref{cor_rhos} and \ref{s_quasi_ine}, and the fact that the inverse mapping $f^{-1}$ of a $K$-quasiconformal mapping $f$ is another $K$-quasiconformal mapping.
\end{proof}

\begin{corollary}\label{cor_5.8}
If $f:\uhp^2\to\uhp^2=f(\uhp^2)$ is a $K$-quasiconformal homeomorphism,
\begin{align*}
\frac{1}{c(K)^K}s_{\uhp^2}(x,y)^K\leq s_{\uhp^2}(f(x),f(y))\leq c(K)s_{\uhp^2}(x,y)^{1\slash K}.    
\end{align*}
\end{corollary}
\begin{proof}
Follows from Corollary \ref{cor_quasi_s_bothways}, or directly from Corollary \ref{s_quasi_ine} and the fact that, in the upper half-plane, th$(\rho_{\uhp^2}(x,y)\slash2)=s_{\uhp^2}(x,y)$.
\end{proof}

\begin{remark}
Note that the inequality in Corollary \ref{cor_5.8} reduces to an identity if $K=1$.
\end{remark}

\begin{corollary}
If $\theta\in(0,2\pi)$ and $f:S_0\to S_\theta=f(S_0)$ is a $K$-quasiconformal homeomorphism, the following inequalities hold for all $x,y\in S_0$.
\begin{align*}
&(1)\quad 
\frac{\theta}{c(K)^K\pi\sin(\theta\slash2)}s_{S_0}(x,y)^K\leq s_{S_\theta}(f(x),f(y))\leq c(K)\left(\frac{\pi}{2}\right)^{1\slash K}s_{S_0}(x,y)^{1\slash K}\\
&\quad\quad\text{if }\theta\in(0,\pi),\\
&(2)\quad
\frac{1}{c(K)^K}s_{S_0}(x,y)^K\leq s_{S_\theta}(f(x),f(y))\leq c(K)\left(\frac{\pi}{2}\right)^{1\slash K}s_{S_0}(x,y)^{1\slash K}\text{ if }\theta=\pi,\\
&(3)\quad
\frac{1}{c(K)^K}s_{S_0}(x,y)^K\leq s_{S_\theta}(f(x),f(y))\leq\frac{c(K)\theta}{\pi}\left(\frac{\pi}{2}\right)^{1\slash K}s_{S_0}(x,y)^{1\slash K}\text{ if }\theta\in[\pi,2\pi).
\end{align*}
\end{corollary}
\begin{proof}
Follows from Corollary \ref{cor_rhos}, Theorem \ref{strip_rhos} and Corollary \ref{s_quasi_ine}.
\end{proof}

\begin{lemma}
If $\alpha,\beta\in(0,\pi]$ and $f:S_\alpha\to S_\beta$, $f(z)=z^{(\beta\slash \alpha)}$, then for all $x,y\in S_\alpha$
\begin{align*}
&s_{S_\alpha}(x,y)\leq s_{S_\beta}(f(x),f(y))\leq\frac{\beta\sin(\alpha\slash2)}{\alpha\sin(\beta\slash2)}s_{S_\alpha}(x,y)\text{ if }\alpha\leq\beta,\\
&\frac{\beta\sin(\alpha\slash2)}{\alpha\sin(\beta\slash2)}s_{S_\alpha}(x,y)\leq s_{S_\beta}(f(x),f(y))\leq s_{S_\alpha}(x,y)\text{ otherwise.}
\end{align*}
Furthermore, the constants here are sharp.
\end{lemma}
\begin{proof}
By symmetry, we can suppose that $\arg(x)\leq\arg(y)$ and $\arg(x)\leq\alpha-\arg(y)$. It follows that $\inf_{z\in S_\alpha}(|x-z|+|z-y|)=|\overline{x}-y|$ and $\inf_{z\in S_\beta}(|f(x)-z|+|z-f(y)|)=|\overline{f(x)}-f(y)|$. Consider now the quotient 
\begin{align}\label{conf_s_quo}
\frac{s_{S_\beta}(f(x),f(y))}{s_{S_\alpha}(x,y)}
=\frac{|f(x)-f(y)||\overline{x}-y|}{|\overline{f(x)}-f(y)||x-y|}
=\frac{\text{th}(\rho_{\uhp^2}(f(x),f(y))\slash2)}{\text{th}(\rho_{\uhp^2}(x,y)\slash2)}.   
\end{align}
By Lemma \ref{lemma_simplifying_rhosquo}, there exits a conformal mapping $h:S_\alpha\to S_\alpha$ such that $h(x)=e^{(1-k)\alpha i\slash2}$ and $h(y)=e^{(1+k)\alpha i\slash2}$ for some $k\in(0,1)$, and the quotient \eqref{conf_s_quo} is invariant to this transformation. Thus, we can fix $x=e^{(1-k)\alpha i\slash2}$ and $y=e^{(1+k)\alpha i\slash2}$. Now, $f(x)=e^{(1-k)\beta i\slash2}$ and $f(y)=e^{(1+k)\beta i\slash2}$, so it follows that the quotient \eqref{conf_s_quo} is
\begin{align*}
Q(k,\alpha,\beta)\equiv\frac{\sin(k\beta\slash2)\sin(\alpha\slash2)}{\sin(k\alpha\slash2)\sin(\beta\slash2)}.
\end{align*}
By differentiation, it can be proved that this quotient is monotonic with respect to $k$ and its extreme values are
\begin{align*}
\lim_{k\to0^+}Q(k,\alpha,\beta)=\frac{\beta\sin(\alpha\slash2)}{\alpha\sin(\beta\slash2)}
\quad\text{and}\quad
\lim_{k\to1^-}Q(k,\alpha,\beta)=1.
\end{align*}
It only depends on whether $\alpha\leq\beta$ or not, which one of these extreme values is the minimum and which one the maximum, so our theorem follows.
\end{proof}

\renewcommand{\refname}{References} 


\end{document}